\documentclass[a4paper,11pt]{amsart}
\usepackage{amssymb}
\usepackage{xypic}
\usepackage[only,mapsfrom]{stmaryrd}
\usepackage{graphicx}

\newtheorem{theorem}{Theorem}[section]
\newtheorem{lemma}[theorem]{Lemma}
\newtheorem{conj}[theorem]{Conjecture}
\newtheorem{cor}[theorem]{Corollary}

\theoremstyle{definition}

\newtheorem{example}[theorem]{Example}

\theoremstyle{remark}

\numberwithin{equation}{section}

\DeclareMathOperator{\diam}{diam}
\DeclareMathOperator{\characteristic}{char}

\title[{Torus actions, curvature and $\pi_1$}]{Torus actions, almost non-negative curvature and fundamental groups}

\author{Michael Wiemeler}
\address{Mathematisches Institut\\ Universit\"at M\"unster\\Einsteinstraße 62\\D-48149 M\"unster \\Germany}
\email{wiemelerm@uni-muenster.de}
\thanks{}



\date{October 3, 2026}

\begin{document}
\begin{abstract}
 We show that purely topological torus symmetries impose strong restrictions on the fundamental groups of almost non-negatively curved closed manifolds and on their actions on higher homotopy groups, even when the torus action is unrelated to the metric.\end{abstract}

\maketitle


\section{Introduction}

The study of Riemannian manifolds with curvature bounded from below has long been an active field of research.

For example, for closed manifolds of almost non-negative sectional curvature Fukaya and Yamaguchi \cite{MR1185120} showed that the fundamental group contains a nilpotent subgroup of finite index and a solvable subgroup whose index can be bounded above by a constant which only depends on the dimension of the manifold.
Later Kapovitch, Petrunin and Tuschmann \cite{MR2630041} proved that also the index of the nilpotent subgroup can be bounded above by a constant which only depends on the dimension of the manifold.
Moreover it has been shown by Gromov \cite{gromov_almost_1978} that the minimal number of generators of the fundamental group of the manifold can be bounded above by a constant which only depends on the dimension of the manifold.

In another paper \cite{gromov_curvature_1981} he also showed that for any field coefficients the sum of Betti numbers of an almost non-negatively curved manifold is bounded above by a constant which only depends on the dimension of the manifold.
Although this result is of a somewhat different nature, it provides another restriction on the topology of these manifolds.

In their papers Fukaya--Yamaguchi stated the following far-reaching conjecture on fundamental groups of non-negatively curved manifolds.

\begin{conj}[{\cite{MR1185120}}]
  \label{sec:introduction-4}
  The fundamental group of a complete non-negati\-ve\-ly curved \(m\)-manifold is \(C(m)\)-abelian, i.e. it contains an abelian subgroup of index at most \(C(m)\), where \(C(m)<\infty\) is a constant which only depends on the dimension \(m\).
\end{conj}

Note that recently a proof of Conjecture~\ref{sec:introduction-4} in dimensions up to four has been given by Bru\`e, Naber and Semola \cite{MR5013747}.
Moreover, it is known that under the assumptions of Conjecture~\ref{sec:introduction-4} the fundamental group of \(M\) contains an abelian subgroup of finite index \cite{MR0303460}.

Here we prove Conjecture \ref{sec:introduction-4} under certain topological symmetry assumptions.

Before we state our main theorem let us recall the definition of an almost non-negatively curved manifold.
A closed smooth manifold \(M\) is said to admit almost non-negative sectional curvature, if there is a sequence of Riemannian metrics \(g_i\), \(i\in \mathbb{N}\), on \(M\) such that
\begin{align*}
  K_{g_i}&\geq -1/i&\diam(M,g_i)&\leq 1,
\end{align*}
where \(K_{g_i}\) denotes the sectional curvature of \(g_i\).

Note that if there is a continuous action of a torus \(T\) on a manifold \(M\) then the orbit space \(M/T\) is stratified by connective orbit types.

\begin{theorem}
  \label{sec:introduction-2}
  Let \(M\) be an almost non-negatively curved closed smooth manifold of dimension \(m\in \mathbb{N}\).
  Assume there is a continuous torus action on \(M\) such that \(M/T\) has a stratum of dimension at most one.
  If there is no zero-dimensional stratum, assume that the action is locally smooth.
  
  Then Conjecture~\ref{sec:introduction-4} holds for \(M\).
\end{theorem}

We remark that in the above theorem the torus action is not assumed to be isometric.
It does not even have to be smooth.
While manifolds with isometric torus actions and lower curvature bounds are well studied (see for example  \cite{searle_how_2015},  \cite{wiemeler}, \cite{harvey_almost_2020}, \cite{MR4333977}, \cite{MR4778058}), non-isometric actions on such spaces are rarely studied.

Note that an isolated fixed point of a torus action gives rise to a zero-dimensional stratum in the orbit space.
In this special case we get the following stronger corollary.

\begin{cor}
  \label{sec:introduction-1}
  Let  \(M\) be an almost non-negatively curved closed smooth manifold of dimension \(m\in \mathbb{N}\).
  If there is a continuous torus action which has at least one isolated fixed point, then the fundamental group of \(M\) has order smaller than \(C(m)\), where \(C(m)\) is a constant only depending on the dimension of \(M\).
\end{cor}

In the case of a zero-dimensional stratum we also get the following result which gives an affirmative answer to a conjecture of Kapovitch--Petrunin--Tuschmann {\cite{MR2630041}} under symmetry assumptions.

\begin{theorem}
  \label{sec:introduction-7}
  Let \(M\) be an almost non-negatively curved closed smooth manifold.
  Assume there is a continuous torus action on \(M\) such that there is a zero-dimensional stratum \(S\) in the orbit space.
  Then a finite index subgroup \(A\subset \pi_1(M)\) acts trivially on \(\bigoplus_{k\geq 2} \pi_k(M)\).
  Moreover \(A\) is contained in the center of \(\pi_1(M)\).
  If the points in the orbit of the \(T\)-action corresponding to \(S\) have connected isotropy group, then the index of \(A\) can be bounded above by a constant \(C(m)\) which only depends on the dimension of \(M\).
\end{theorem}

It is not known how restrictive the assumption of the existence of a continuous torus action is.
On the one hand it is expected that most manifolds have little symmetry.
On the other hand there are only a few known examples of manifolds for which it has been shown that they do not admit continuous \(S^1\)-actions.
Besides certain manifolds related to aspherical manifolds the only class of examples I know is a class of simply connected six-manifolds constructed by Puppe \cite{puppe_simply_1988}.
For an overview on this and related questions see \cite{puppe_manifolds_2007}.
For our context we note that all aspherical non-negatively curved manifolds are flat and therefore satisfy Conjecture~\ref{sec:introduction-4} by the Bieberbach Theorems.

Slice maximal torus actions are examples of torus actions having zero-dimensional strata in their orbit spaces.
Here a smooth effective action of a torus \(T\) on a manifold \(M\) is called slice maximal, if there is a point \(x\in M\) such that the dimensions of \(M\), of the orbit \(Tx\) of \(x\) and of the isotropy group \(T_x\)  of \(x\) are related as follows:
\[\dim M=\dim Tx+2\dim T_x.\]
Note that in this situation the orbit \(Tx\) is a fixed point component of the \(T_x\)-action on \(M\) and the isotropy group \(T_x\) is connected.

Slice-maximal torus actions were studied in \cite{MR3956693} and \cite{MR3482597}  under the name of maximal torus actions.
They were renamed in \cite{galaz-garcia_torus_2018} to be more specific about the sense in which the torus action is maximal.
Simply connected, closed, non-negatively curved manifolds with isometric (almost) slice maximal torus actions have been studied in \cite{MR4333977} and \cite{MR4778058} where they are called isotropy maximal.

As mentioned earlier, isolated fixed points are also examples of zero-dimensional strata.
Prominent examples of \(T\)-manifolds with isolated fixed points are torus manifolds.
A torus manifold is a closed oriented smooth manifold of dimension \(2n\) with a smooth effective action of an \(n\)-dimensional torus such that there are fixed points.
They are well studied generalizations of closed smooth toric varieties (see \cite{buchstaber_toric_2015}).
Torus manifolds with invariant metrics of non-negative sectional curvature have been classified in \cite{wiemeler}.

\begin{example}
  \label{sec:introduction-3}
  In \cite[Corollary 2.3]{wiemeler_exotic_2013} for every finitely presentable group \(\pi\) a construction of families of torus manifolds \(M\) with \(\pi_1(M)=\pi\) was given which works in every even dimension at least eight.
  Corollary~\ref{sec:introduction-1} allows to show that in a fixed dimension there is no metric of almost non-negative sectional curvature on these manifolds if the order of \(\pi\) is large enough.
  This is in particular the case for large cyclic groups.
  In this latter special case it is not possible to deduce non-existence of such metrics from previously known obstructions  to the existence of almost non-negatively curved metrics depending only on the fundamental groups  such as \cite{MR1185120}, \cite{MR2630041}, \cite{gromov_almost_1978}.
  However, as we will see from the proof of Corollary~\ref{sec:introduction-1}, the universal covering of \(M\) does not satisfy Gromov's Betti number estimate \cite{gromov_curvature_1981}  for large \(\pi\). (See Section~\ref{sec:coroll-refs-1} for details.)
\end{example}

A large class of manifolds with torus actions and at least two isolated fixed points can also be constructed as follows.
Let \(M\) be a \(2n\)-dimensional smooth manifold with a smooth effective action of a torus \(T\) of dimension at most \(n\) and \(S^{2n}\) be the \(2n\)-dimensional standard sphere with a linear effective action of \(T\) such that \(S^{2nT}\) consists of exactly two points.
Then one can form the equivariant connected sum \(N^{2n}=M^{2n}\#_T S^{2n}\) at principal orbits.
The resulting smooth \(T\)-manifold \(N^{2n}\) has at least two isolated fixed points.

When we restrict ourselves to non-negatively curved closed manifolds with isometric torus action satisfying some condition on minimal strata we get the same conclusion as in Theorem~\ref{sec:introduction-2} for actions having orbit spaces with strata of dimension at most three (see Theorem~\ref{sec:almost-non-negative-7}).

  It is  natural to ask whether the conclusion of Conjecture~\ref{sec:introduction-4} also holds for complete Riemannian manifolds of non-negative or positive Ricci curvature.
  However examples which show that this is not the case have been given by Wei \cite{MR940494} in the non-compact complete case and more recently by Bru\`e, Naber and Semola \cite{brue_compact_2026} in the closed case.
  We note in Section~\ref{sec:example-BNS} that the examples constructed in \cite{brue_compact_2026} admit effective \(T^2\)-actions.
  As an application of our methods we show that these examples do not admit almost non-negative sectional curvature.
This also shows that the proof of Theorem~\ref{sec:introduction-2} cannot be generalized directly to positively Ricci curved manifolds.

The main new ingredient in the proofs of Theorems~\ref{sec:introduction-2} and~\ref{sec:introduction-7} is an upper bound for the index of the image of the natural map \(\pi_1(F)\rightarrow \pi_1(M)\), where \(F\) is a \(T\)-fixed point component in the \(T\)-manifold \(M\) in terms of the Betti numbers of the universal covering of \(M\).
Note that a similar bound for the case that \(M\) is closed and has finite fundamental group has been given by Rong \cite{MR1713297}.
However we do not assume that \(\pi_1(M)\) is finite.

If \(M\) is (almost) non-negatively curved then the sum of these Betti numbers is bounded above
by a constant which only depends on the dimension of the manifold.
This follows from Gromov's Betti number estimate \cite{gromov_curvature_1981} for (almost) non-negatively curved manifolds and the fibration theorem of Kapovitch, Petrunin and Tuschmann \cite{MR2630041} for manifolds with almost non-negative sectional curvature.

The proofs of these theorems are then completed by studying the fundamental groups of manifolds \(S\) with torus actions of cohomogeneity at most three.
For the case of \(\dim S/T>1\) arguments from the proof of  Bru\`e, Naber and Semola \cite{MR5013747} enter here.

Since by the splitting theorem \cite{MR0303460} the universal covering of a closed non-negatively Ricci curved Riemannian manifold has the homotopy type of a closed manifold, parts of the argument apply also in the setting of non-negatively Ricci curved manifolds (see Section \ref{sec:non-negative-ricci}).
Note here that by the counterexample to the Milnor conjecture in \cite{brue_fundamental_2025} in general complete manifolds of non-negative Ricci curvature do not have to have the homotopy type of  compact manifolds, in contrast to the case of complete manifolds of non-negative sectional curvature.
  
This paper is structured as follows.
In Section~\ref{sec:main-technical-lemma-1} we prove our main technical lemma.
In Section~\ref{sec:almost-non-negative-2} we prove slightly more general versions of Theorem~\ref{sec:introduction-2}, Corollary~\ref{sec:introduction-1} and Theorem \ref{sec:introduction-7}.
In Section~\ref{sec:non-negative-ricci} we discuss manifolds of non-negative Ricci curvature.
Finally in Section~\ref{sec:examples} we discuss the construction of Example~\ref{sec:introduction-3} and other examples in more detail.

I would like to thank Christoph Böhm and Burkhard Wilking for discussions on the subject of this paper and Anand Dessai and Wilderich Tuschmann for comments on an earlier version of this paper.

Feedback from ChatGPT has been used for literature search, improving the presentation of the paper and to strengthen Lemma \ref{sec:main-technical-lemma-2}.

The research for this paper was funded by the Deutsche Forschungsgemeinschaft (DFG, German Research Foundation) under Germany's Excellence Strategy EXC 2044/2 –390685587, Mathematics Münster: Dynamics–Geometry–Structure and through CRC 1442, Geometry: Deformations and Rigidity at University of Münster.

\section{The main technical lemma}
\label{sec:main-technical-lemma-1}

Our main technical lemma is a version of a lemma from \cite{MR1713297}.
It has been shown in \cite{MR1713297} for the case that \(M\) is closed and has finite fundamental group.
However, we do not make any assumptions on \(\pi_1(M)\).

\begin{lemma}
  \label{sec:main-technical-lemma}
  Let \(M\) be a topological manifold with a continuous action of a torus \(T\) such that \(M^{T}\neq \emptyset\)
  and the universal covering \(\tilde{M}\) of \(M\) has the homotopy type of a
  compact manifold. Let, moreover, \(F\) be a component of
  \(M^{T}\).
  Then the image of the natural map \(\pi_1(F)\rightarrow \pi_1(M)\)
  has finite index at most \(\dim H^*(\tilde{M};\mathbb{Q})\).
\end{lemma}
\begin{proof}
  Since there are \(T\)-fixed points the \(T\)-action lifts to an
  action on the universal covering of \(M\).
  Because the universal covering of \(M\) has the homotopy type of a
  compact manifold, the \(T\)-action on \(\tilde{M}\) has finitely many orbit types by  \cite[Theorem 3.5]{mann_finite_1962}.
  Therefore there is a sufficiently large prime \(p\in \mathbb{Z}\) such that
  \begin{itemize}
  \item \(\tilde{M}^{T}=\tilde{M}^{G}\) and
  \item \(\dim H^k(\tilde{M};\mathbb{Z}_p)=\dim H^k(\tilde{M};\mathbb{Q})\) for all \(k\in \mathbb{N}\),
  \end{itemize}
  where \(G=\mathbb{Z}_p\) is a suitable cyclic subgroup of \(T\) of order \(p\). 

  Hence it follows from \cite[Theorem VII.2.1]{bredon_introduction_1972}  that
  \begin{equation*}
    \dim H^*(\tilde{M}^{T};\mathbb{Z}_p)=\dim H^*(\tilde{M}^{G};\mathbb{Z}_p)\leq \dim H^*(\tilde{M};\mathbb{Z}_p)=\dim H^*(\tilde{M};\mathbb{Q}).
  \end{equation*}
  So the number of fixed point components of the \(T\)-action on \(\tilde{M}\) is bounded above by \(\dim H^*(\tilde{M};\mathbb{Q})\).
  
  Since the \(T\)-action on \(\tilde{M}\) commutes with the action
  of the deck transformation group, the subgroup of the deck
  transformation group which leaves one of the fixed point components
  invariant has finite index \(\leq \dim H^*(\tilde{M};\mathbb{Q})\).
  This subgroup can be identified with the image of the natural map
  \(\pi_1(F)\rightarrow \pi_1(M)\) after identifying the deck
  transformation group with \(\pi_1(M)\).
\end{proof}

As an immediate corollary to Lemma~\ref{sec:main-technical-lemma} we get the following result which might be seen as a partial generalization of the theorem that a \(S^1\)-action on a closed aspherical manifold is either fixed-point-free or trivial \cite{donnelly_compact_1982}.

\begin{cor}
  Let \(M_1\) be a simply connected closed manifold and \(M_2\) an aspherical manifold.
  Assume that there is a continuous \(S^1\)-action on \(M_1\times M_2\) such that \((M_1\times M_2)^{S^1}\neq \emptyset\).
  Let \(F\) be a component of \((M_1\times M_2)^{S^1}\).
  Then the image of \(\pi_1(F)\) in \(\pi_1(M_2)=\pi_1(M_1\times M_2)\) has finite index.
\end{cor}
\begin{proof}
  Since \(M_2\) is aspherical, its universal covering is contractible.
  Hence the universal covering of \(M_1\times M_2\) is homotopy equivalent to \(M_1\).
  Therefore the claim follows from Lemma~\ref{sec:main-technical-lemma}.
\end{proof}

For the proof of Theorem~\ref{sec:introduction-7}, we also need the following lemma.

\begin{lemma}
\label{sec:main-technical-lemma-2}
  Let \(M\) be a connected topological manifold on which a torus \(T\) acts continuously.
  Assume there is a point \(x\in M\) whose isotropy group \(T_x\) is finite.
  Then a subgroup \(A\) of index \(|T_x|<\infty\) of the fundamental group of the orbit \(Tx\) of \(x\) acts trivially on \(\bigoplus_{k\geq 2}\pi_k(M)\).
  Moreover the image of \(A\) in \(\pi_1(M)\) is contained in the center of \(\pi_1(M)\).
\end{lemma}
\begin{proof}
  Let \(\tilde{M}\) be the universal covering of \(M\).
  Then the action of \(\pi_1(M)\) on \(\pi_k(M)\), \(k\geq 2\), can be identified with the action of the deck transformation group of \(\tilde{M}\rightarrow M\) on \(\pi_k(\tilde{M})\cong \pi_k(M)\).
  For this to make sense note that since \(\tilde{M}\) is simply connected, its homotopy groups are up to canonical isomorphism independent of the choice of a base point in \(\tilde{M}\).
  Hence all groups of homeomorphisms of \(\tilde{M}\) act on \(\pi_k(\tilde{M})\).

  The \(T\)-action on \(M\) lifts to an action of the universal covering group \(\tilde{T}\) of \(T\) on \(\tilde{M}\).
  The action of \(\pi_1(T)\) induced from \[\pi_1(T)\rightarrow \pi_1(Tx)\rightarrow \pi_1(M)\]
  can be identified with the action of a subgroup of \(\tilde{T}\) on \(\tilde{M}\).
  Since \(\tilde{T}\) is connected, the action of \(\pi_1(T)\) on \(\pi_k(M)\) is trivial.
  Moreover, because \(T_x\) is finite, the natural map \(T\rightarrow Tx\) is a finite \(|T_x|\)-sheeted covering.
  This proves the first claim.

  To prove the second claim let \(g:[0,1]\rightarrow T\) be a path with \(g(0)=g(1)=1\).
  Let moreover \(\gamma:[0,1]\rightarrow M\) be a loop with base point \(x\).
  Then \(H:[0,1]\times [0,1]\rightarrow M\), \(H(s,t)=g(s)\gamma(t)\) is a continuous map from the two-disk to \(M\) whose restriction to the boundary is the commutator of \(g(\cdot)x\) and \(\gamma\) in \(\pi_1(M;x)\).
  Therefore the image of \(\pi_1(T)\) is contained in the center of \(\pi_1(M)\).
\end{proof}

\section{Almost non-negative sectional curvature}
\label{sec:almost-non-negative-2}

In this section we prove Theorem \ref{sec:introduction-2} and Theorem~\ref{sec:introduction-7}.
We start with a combination of slightly more general versions of Theorem~\ref{sec:introduction-2} for the case of a zero-dimensional stratum and Theorem~\ref{sec:introduction-7}.

To state it properly we first recall the definition of the strata of the orbit space.
So assume the torus \(T\) acts continuously on the manifold \(M\).
For a connected subgroup \(H\) of \(T\) we call the connected components of the set
\[\{Tx\in M/T;\;(T_x)^0=H\}\]
of those points in \(M/T\), whose isotropy group has identity component \(H\), the open \(H\)-strata. Their closures are the closed strata.
Note that the closed strata are partially ordered by inclusion.

\begin{theorem}
  \label{sec:almost-non-negative}
  Let \(M\) be an almost non-negatively curved closed smooth manifold of dimension \(m\in \mathbb{N}\) and \(N\) a topological manifold with continuous action of a torus \(T\) such that there is a zero-dimensional stratum \(S\) in \(N/T\).
  Assume that \(M\) and \(N\) are homotopy equivalent.

  Then \(\pi_1(M)\cong \pi_1(N)\) is \(C(m)\)-abelian. Moreover, there is a finite index subgroup \(A\subset \pi_1(M)\) which acts trivially on \(\bigoplus_{k\geq 2}\pi_k(M)\).
  Furthermore, \(A\) is contained in the center of \(\pi_1(M)\).
  If the isotropy group of an orbit in \(S\) is connected, then the index of \(A\) is bounded above by \(C(m)\).
\end{theorem}

For the proof we need the following result due to Kapovitch, Petrunin, and Tuschmann.

\begin{theorem}[{\cite{MR2630041}}]
  \label{sec:almost-non-negative-1}
Let \(M\) be an almost non-negatively curved closed smooth manifold.
Then a finite cover \(\tilde{M}\) of \(M\) is the total space of a fiber bundle
\[\mathcal{F}\rightarrow \tilde{M}\rightarrow \mathcal{N}\]
over a nilmanifold \(\mathcal{N}\) with a simply connected fiber \(\mathcal{F}\). Moreover, the fiber \(\mathcal{F}\) is
almost nonnegatively curved in a generalized sense for which Gromov's Betti number estimate holds.
\end{theorem}

Note that the universal covering of a \(k\)-dimensional nilmanifold is diffeomorphic to \(\mathbb{R}^k\).

\begin{proof}[Proof of Theorem \ref{sec:almost-non-negative}] 
  By Theorem~\ref{sec:almost-non-negative-1}, the universal covering of \(M\) is diffeomorphic to \(\mathcal{F}\times \mathbb{R}^k\), where the total Betti number of \(\mathcal{F}\) is bounded from above by some constant \(C(m)\) depending only on the dimension of \(M\).

  Since there is a zero dimensional stratum in \(N/T\), there is a subtorus \(H\subset T\) and a component \(\hat{S}\) of \(N^H\) such that \(T/H\) acts transitively on \(\hat{S}\).
  So \(\hat{S}\) is homeomorphic to a torus and therefore has abelian fundamental group.
  Hence the first claim follows from Lemma~\ref{sec:main-technical-lemma}.
  
  To prove the other claims we can assume that \(H\) is the maximal subtorus of \(T\) which acts trivially on \(\hat{S}\).

  To prove the second claim, let \(H'\subset T\) be a complementary subtorus to \(H\), i.e. \(T=H\times H'\).
  Then the isotropy group \(H'_x\) of every point \(x\in \hat{S}\) is finite.
  Hence, by Lemma~\ref{sec:main-technical-lemma-2}, there is an abelian finite index subgroup of \(\pi_1(\hat{S})\) which acts trivially on \(\bigoplus_{k\geq 2}\pi_k(N)\).
  Moreover, the image of this subgroup in \(\pi_1(N)\) is contained in the center of \(\pi_1(N)\).
  Since by the first part of the proof the image of \(\pi_1(\hat{S})\) has finite index in \(\pi_1(M)\) the second claim follows.

  If for \(x\in \hat{S}\) the isotropy group \(T_x\) is connected, i.e. \(T_x=H\),
  then \(H'_x\) is the trivial group.
  So the third claim follows in the same way from Lemma~\ref{sec:main-technical-lemma-2} as the second.
\end{proof}

This proof in particular shows the following version of Corollary~\ref{sec:introduction-1}.

\begin{cor}
  \label{sec:almost-non-negative-3}
   Let \(M\) be an almost non-negatively curved closed smooth manifold of dimension \(m\in \mathbb{N}\) and \(N\) a topological manifold with continuous action of a torus \(T\) such that \(N^{T}\) contains an isolated fixed point. Assume \(M\) and \(N\) are homotopy equivalent.  Then the order of \(\pi_1(M)\) is bounded above by \(C(m)\).
\end{cor}
\begin{proof}
  If there is an isolated fixed point \(x\in N^T\), then in the proof of Theorem~\ref{sec:almost-non-negative} we can choose \(\hat{S}=\{x\}\).
  Therefore the claim follows.
\end{proof}

Using a result of Böhm and Wilking we can show the following result.

\begin{cor}
  Let \(M\) be a closed smooth manifold which admits a continuous \(S^1\)-action with at least one isolated fixed point and a (not necessarily invariant) metric of non-negative sectional curvature.

  Then \(M\) admits a metric of positive Ricci curvature and the order of the fundamental group of \(M\) is bounded above by a constant which only depends on the dimension of \(M\).
\end{cor}
\begin{proof}
  By Corollary~\ref{sec:almost-non-negative-3} the claim about the fundamental group of \(M\) is clear.
  Therefore by a result of B\"ohm and Wilking \cite{boehm07:_nonneg_ricci}, the metric of non-negative sectional curvature can be smoothed using the Ricci flow to a metric of positive Ricci curvature.
\end{proof}

If \(M\) is a torus manifold with an invariant metric of non-negative sectional curvature, then the above corollary also follows from the classification result given in \cite{wiemeler}. But the proof presented here is much simpler and less technical.

Now we turn to locally smooth actions on manifolds.
For background on this type of actions see Chapter IV of \cite{bredon_introduction_1972}.
We prove a slightly more general version of Theorem~\ref{sec:introduction-2} for the case of a one-dimensional stratum. 

\begin{theorem}
  \label{sec:almost-non-negative-4}
  Let \(M\) be an almost non-negatively curved closed smooth manifold of dimension \(m\), which is homotopy equivalent to a topological manifold \(N\) on which a torus acts locally smoothly such that there is a one-dimensional stratum in \(N/T\).

  Then the fundamental group of \(M\) is \(C(m)\)-abelian.
\end{theorem}
\begin{proof}
  Since the \(T\)-action on \(N\) lifts to an action of a finite covering group of \(T\) on the orientation covering \(\tilde{N}\) of \(N\) and since \(\pi_1(\tilde{N})\) has index two in \(\pi_1(N)\) it suffices to show the theorem for the case that \(N\) is orientable.

  Then the \(T\)-action on \(N\) is orientation preserving, since \(T\) is connected.
  Since there is a one-dimensional stratum in \(N/T\), there is a submanifold \(S\) of \(N\) such that
  \begin{itemize}
  \item \(S\) is a fixed point component of the action of a subtorus \(H\subset T\),
  \item  \(T\) acts on \(S\) with cohomogeneity-one, i.e. with one-dimensional orbit space.
  \end{itemize}

  Note that by the first point \(S\) is orientable (see Theorem IV.2.1 in \cite{bredon_introduction_1972}).
 
  If \(S\) is non-compact then by Theorem IV.8.1 in \cite{bredon_introduction_1972}, \(S\) is homotopy equivalent to an orbit of the \(T\)-action on \(N\).
  Therefore in this case our claim follows from Lemma~\ref{sec:main-technical-lemma}.

  If \(S\) is compact then by Theorem IV.8.2 in \cite{bredon_introduction_1972}, there are two cases to consider:
  \begin{enumerate}
  \item \(S/T=S^1\),
  \item \(S/T=[-1,+1]\).
  \end{enumerate}
  In the first case \(S\) is a principal torus bundle over \(S^1\) hence homeomorphic to a torus.
  So in this case our claim follows from Lemma~\ref{sec:main-technical-lemma}.

  In the second case all orbits in \(]-1,+1[\) have the same isotropy group \(K_0\).
  The orbits \(\pm 1\) have strictly larger isotropy groups \(K_\pm\supset K_0\) such that
  \(K_{\pm}/K_0\) is \(S^{k_{\pm}}\) with \(k_\pm=0\) or \(k_\pm=1\).

  If \(k_{\pm}=0\) then the action of \(T\) on \(S\) is not orientation preserving in a neighborhood of the orbit \(\pm 1\). So we get a contradiction  in this case.
  But if \(k_{\pm}=1\), then there is a zero-dimensional stratum in \(N/T\).
  Hence our claim follows from Theorem~\ref{sec:almost-non-negative} in this case.

  So the theorem is proved.
\end{proof}

We next prove a theorem about isometric actions.
Before we do so we introduce some notation.
First note that isometric actions of compact Lie groups on smooth Riemannian manifolds are smooth.

\begin{theorem}
  \label{sec:almost-non-negative-7}
  Let \(m\in \mathbb{N}\). Then there is a constant \(0<C(m)<\infty\) such that if \(M\) is a non-negatively curved closed smooth Riemannian manifold of dimension \(m\) with an isometric effective action of a torus \(T\) such that there is a minimal closed stratum \(S\subset M/T\) with
  \begin{enumerate}
  \item\label{item:3} all \(Tx\in S\) have the same isotropy group \(H\subset T\),
  \end{enumerate}
  then the following holds:
  \begin{enumerate}
  \item[(a)] If \(\dim S\leq 3\) then \(\pi_1(M)\) contains an abelian subgroup of index at most \(C(m)\).
  \item[(b)] If \(\dim S=4\) then \(\pi_1(M)\) contains a two-step nilpotent subgroup of index at most \(C(m)\).
  \end{enumerate}
\end{theorem}

\begin{proof}
  Since \(M\) is closed there are only finitely many \(T\)-orbit types in \(M\).
  Therefore there is an \(S^1\subset T\) such that \(\hat{S}\) is a fixed point component of \(S^1\) or \(\hat{S}=M\), where \(\hat{S}\subset M\) is the preimage of \(S\) under the orbit map.
  Hence by Lemma~\ref{sec:main-technical-lemma} and Gromov's Betti number estimate, the claim would follow if we can show that \(\pi_1(\hat{S})\) is uniformly virtually abelian or two-step nilpotent, respectively.

  Since \(S\) is a minimal stratum, by Assumption~(\ref{item:3}), \(\hat{S}\) is a principal \(T/H\)-bundle over the closed manifold \(S\), where \(H\) is the common isotropy group of the points in \(\hat{S}\).
  Therefore there is a central extension
  \[\pi_1(T/H)\rightarrow \pi_1(\hat{S})\rightarrow \pi_1(S)\rightarrow 1.\]

  Hence, to see that \(\pi_1(\hat{S})\) contains a two-step nilpotent subgroup of controlled index it suffices to show that there is an abelian subgroup of \(\pi_1(S)\) whose index is bounded above by a constant which is independent of \(M\).

  Note that the submanifold \(\hat{S}\) is totally geodesic in \(M\). So it is non-negatively curved. Hence \(S\) equipped with the quotient metric is also non-negatively curved.
  So our claim for \(\dim S=4\) follows from the main result of \cite{MR5013747}.

  If \(\dim S\leq 1\), then it follows from Theorems~\ref{sec:almost-non-negative} and~\ref{sec:almost-non-negative-4} that \(\pi_1(M)\) contains an abelian subgroup of index less than \(C(m)\).

  Next we consider the case where \(\dim S=2\).
  Then by the classification of surfaces and the Gauß--Bonnet theorem \(S\) is diffeomorphic to the sphere \(S^2\), the torus \(T^2\), the real projective plane \(\mathbb{R} P^2\) or the Klein bottle \(K\).
  In the last two cases \(\hat{S}\) is non-orientable and the orientation double covering \(\tilde{S}\)  of \(\hat{S}\) is diffeomorphic to a principal \(T/H\)-bundle over \(S^2\) or \(T^2\).
  Since \(\pi_1(\tilde{S})\) has index two in \(\pi_1(\hat{S})\), it suffices to show that if \(R\) is a principal torus bundle over \(S^2\) or \(T^2\) equipped with a torus invariant non-negatively curved Riemannian metric, then \(\pi_1(R)\) is abelian.

  Since \(S^2\) is simply connected the first case is clear.
  For the second case we note that there is a tower
  \[R=E_k\rightarrow E_{k-1}\rightarrow\dots\rightarrow E_1\rightarrow E_0=T^2,\]
  such that each \(E_l\rightarrow E_{l-1}\), \(1\leq l\leq k\), is a principal \(S^1\)-bundle and \(E_l\) is equipped with an \(S^1\)-invariant metric of non-negative sectional curvature.

  Therefore in this case the claim for \(\dim S=2\) follows by induction on \(k\) from Lemma~\ref{sec:almost-non-negative-5} below.

  Finally we consider the case \(\dim S=3\).
  In this case it follows from the splitting theorem and Lemma~\ref{sec:almost-non-negative-6} below that the universal covering of \(\hat{S}\) is isometric to a Riemannian product
  \(N\times \mathbb{R}^k\), where \(N\) is a point, \(S^2\) or \(S^3\) equipped with some Riemannian metric of non-negative sectional curvature and \(\mathbb{R}^k\) is the flat Euclidean space of dimension \(k=\dim \hat{S}-\dim N\leq \dim \hat{S}\leq \dim M=m\).

  The arguments in the proof of Theorem 1.1 of \cite{MR5013747}, in particular Claims 2, 3 and 4, work in this situation where the universal covering is isometric to a product of a sphere with flat Euclidean space of bounded dimension.
  Therefore it follows that the fundamental group of \(\hat{S}\) contains an abelian subgroup whose index is bounded above by a constant which only depends on \(m\).
  So our claim is proved in this case.

  Therefore the theorem is proved.
\end{proof}

\begin{lemma}
  \label{sec:almost-non-negative-5}
  Let \(E\rightarrow B\) be a smooth principal \(S^1\)-bundle, such that the total space \(E\) is equipped with a metric of non-negative sectional curvature.
  If \(B\) is diffeomorphic to an \(n\)-dimensional torus \(T^n\), \(n\geq 1\), then \(E\) is diffeomorphic to an \((n+1)\)-dimensional torus \(T^{n+1}\).
\end{lemma}
\begin{proof}
  It follows from the Gysin sequence of the bundle \(E\rightarrow B\) that the first Betti number of \(E\) is \(n\) or \(n+1\).
  In the second case, the Euler class of the bundle vanishes.
  Hence, the bundle is trivial and the claim follows.

  So assume that the first Betti number of \(E\) is \(n\).
  Since the universal covering of \(E\) is diffeomorphic to \(\mathbb{R}^{n+1}\), it follows from the splitting theorem \cite{MR0303460} that \(E\) is a flat manifold.
  Therefore, by the Bieberbach theorems, \(\pi_1(E)\) contains a finite index normal free abelian subgroup \(G\) such that
  the quotient \(\tilde{E}=\mathbb{R}^{n+1}/G\) is isometric to a flat \(n+1\)-dimensional torus.
  Let \(\Gamma\) be the finite group \(\pi_1(E)/G\).
  Then \(\Gamma\) acts on \(\tilde{E}\) such that \(\tilde{E}/\Gamma=E\).
  We have
  \[H^*(E;\mathbb{R})\cong H^*(\tilde{E};\mathbb{R})^{\Gamma}.\]

  Since \[\dim H^1(E;\mathbb{R})=n=\dim H^1(\tilde{E};\mathbb{R})-1,\] it follows that \(\Gamma\) acts with a codimension-one fixed point set on \(H^1(\tilde{E};\mathbb{R})\).
  But this means that \(\Gamma\) acts non-trivially on \(H^{n+1}(\tilde{E};\mathbb{R})\).
  Since
  \[\dim H^{n+1}(\tilde{E};\mathbb{R})=1\]
  it follows that \(H^{n+1}(E;\mathbb{R})=0\) and \(E\) is non-orientable.

  This is a contradiction because \(B\) is orientable and \(E\) as a principal \(S^1\)-bundle over \(B\) is also orientable.
  Therefore we must have \(b_1(E)=n+1\) and the lemma is proved.
\end{proof}

\begin{lemma}
  \label{sec:almost-non-negative-6}
  Let \(S\rightarrow X\) be a smooth principal torus bundle over a closed non-negatively curved three-manifold \(X\).
  Then the universal covering of \(S\) is diffeomorphic to one of the following manifolds:
  \begin{align*}
    S^3&\times \mathbb{R}^k& S^2&\times \mathbb{R}^{k+1}& \mathbb{R}^{k+3}&,
  \end{align*}
  where \(k\) is the dimension of the fiber of \(S\rightarrow X\).
\end{lemma}
\begin{proof}
  The proof is by induction on \(k\) using the following facts:
  \begin{enumerate}
  \item\label{item:7}
    If \(E_k\rightarrow B\) is a principal \(T^k\)-bundle, then there
    is a tower \[E_k\rightarrow E_{k-1}\rightarrow\dots\rightarrow E_1\rightarrow E_0=B,\]
  such that each \(E_l\rightarrow E_{l-1}\), \(1\leq l\leq k\), is a principal \(S^1\)-bundle.
\item\label{item:6} Principal \(S^1\)-bundles over some base manifold \(B\) are classified by \(H^2(B;\mathbb{Z})\). Hence all principal \(S^1\)-bundles over \(S^3\times \mathbb{R}^l\) and \(\mathbb{R}^l\) are trivial.
    \item\label{item:5} The universal covering space of the total space of a non-trivial principal \(S^1\)-bundle over \(S^2\times \mathbb{R}^l\) is diffeomorphic to \(S^3\times \mathbb{R}^l\).
\item\label{item:4} If \(E\rightarrow B\) is some bundle and \(p:\tilde{B}\rightarrow B\) a covering map, then the total space of the pullback of \(E\) along \(p\) is a covering space of \(E\).
\item\label{item:2} By the splitting theorem the universal covering of \(X\) is diffeomorphic to
  \(S^3\), \(S^2\times \mathbb{R}\) or \(\mathbb{R}^3\).
  \end{enumerate}
Indeed, we show by induction on \(k\) that if \(E_k\rightarrow B\) is a principal \(T^k\)-bundle over a manifold \(B\) with universal covering diffeomorphic to
\(S^3\times \mathbb{R}^l\), \(S^2\times \mathbb{R}^{1+l}\) or \(\mathbb{R}^{3+l}\) for some \(l\in \mathbb{N}\), then the universal covering of \(E_k\) is also of this form.

To do so, we look at the tower (\ref{item:7}) for \(E_k\rightarrow B\).
Then, by (\ref{item:6}), (\ref{item:5}) and  (\ref{item:4}) the universal covering of \(E_1\) is diffeomorphic to one of \(S^3\times \mathbb{R}^{l+1}\), \(S^2\times \mathbb{R}^{2+l}\) or \(\mathbb{R}^{4+l}\).
Moreover, \(E_k\) is a principal torus bundle over \(E_1\) with fiber dimension \(k-1\).
Hence the claim follows by induction.

So by (\ref{item:2}) the lemma follows.
\end{proof}

\section{Non-negative Ricci curvature}
\label{sec:non-negative-ricci}

In this section we discuss Riemannian manifolds with \(S^1\)-actions with isolated fixed points and non-negative Ricci curvature.
Our main result in this section is as follows.

\begin{theorem}
  Let \(N\) be a topological \(S^1\)-manifold with at least one isolated fixed
  point and \(M\) be a closed  connected smooth manifold which admits a metric of non-negative
  Ricci-curvature. Assume that \(M\) and \(N\) are homotopy
  equivalent.

  Then the fundamental group of \(M\) (and \(N\)) is finite.
\end{theorem}
\begin{proof}
  This follows from  Lemma~\ref{sec:main-technical-lemma} because the universal covering of
  \(M\) has the homotopy type of a compact manifold by the splitting
  theorem of Cheeger and Gromoll \cite{MR0303460}, as in the proof of Corollary~\ref{sec:almost-non-negative-3}.
\end{proof}

A torus manifold is a closed oriented \(2n\)-dimensional manifold with an effective action of an \(n\)-dimensional torus such that there are fixed points.
They are generalizations of smooth toric varieties.
Since on a torus manifold there is an \(S^1\)-action with only isolated fixed points, the above theorem immediately implies the following corollary.

\begin{cor}
  Let \(M\) be a torus manifold which admits a (not necessarily invariant) metric of non-negative Ricci curvature.
  Then \(\pi_1(M)\) is finite.
\end{cor}

\begin{cor}
\label{sec:torus-manifolds-with}
  Let \(M\) be a \((2n+4)\)-dimensional closed connected smooth manifold with  a  metric of non-negative Ricci curvature.
  Assume that there is no abelian subgroup with at most two generators and finite index in \(\pi_1(M)\).
  If \(M\) admits a (not necessarily isometric) smooth effective action of a torus \(T\) of dimension \(n\) with fixed point, then there is a \(T\)-invariant metric of positive scalar curvature on \(M\).
\end{cor}
\begin{proof}
  Since \(M\) has infinite fundamental group it follows from the splitting theorem that \(\chi(M)=0\).
  Because \(\chi(M^{T})=\chi(M)\) there is a fixed point component \(F\) with \(\chi(F) \geq 0\).
  If \(\dim F < 4\), then \(F\) has finite fundamental group or contains an abelian subgroup of finite index generated by two elements.
  Therefore by Lemma~\ref{sec:main-technical-lemma} and the splitting theorem \(M\) has also such a fundamental group.
  This is not true by assumption.
  Hence, \(\dim F=4\).
  
  Since \(\dim M=2n+4\) it follows from \cite{wiemeler2} that \(M\) admits an invariant metric of positive scalar curvature.
\end{proof}

\section{Examples}
\label{sec:examples}

In this section we discuss examples.
At first we give two applications of Corollary~\ref{sec:introduction-1}.
In the first of these examples it is used as an obstruction to almost non-negative sectional curvature on manifolds with torus actions.
In the second the corollary is used as an obstruction to the existence of certain torus actions on non-negatively curved manifolds.

Finally we review the examples of Bru\`e, Naber and Semola \cite{brue_compact_2026} which showed that Conjecture \ref{sec:introduction-4} does not hold for closed manifolds of non-negative Ricci curvature.

\subsection{{Corollary~\ref{sec:introduction-1} as an obstruction to almost non-negative curvature}}
\label{sec:coroll-refs-1}

In this section we give more details on the construction of Example~\ref{sec:introduction-3}.
Let \(p\in \mathbb{N}\) be a prime number, \(n\in \mathbb{N}\), \(n\geq 1\), \(S\) a (\(4n+2\))-dimensional standard sphere, \(\Sigma\) a lens space of dimension \(2n+1\) with \(\pi_1(\Sigma)= \mathbb{Z}_p\).

Then \(S\) equipped with the action of a maximal torus \(T^{2n+1}\) of \(SO(4n+3)\) is a torus manifold.
We let \(S'\) be the complement of an equivariant tubular neighborhood of a principal orbit of the \(T^{2n+1}\)-action on \(S\).
Moreover, we let \(\Sigma'\) be the complement of a tubular neighborhood of a point in \(\Sigma\).
Then the boundaries of \(S'\) and \(\Sigma'\times T^{2n+1}\) are both equivariantly diffeomorphic to \(S^{2n}\times T^{2n+1}\).
Here the \(T^{2n+1}\)-action on the two products is given by multiplication on the second factor.

So we can equivariantly glue \(S'\) and \(\Sigma'\times T^{2n+1}\) along their boundaries to construct a \((4n+2)\)-dimensional torus manifold \(M\).
By the proof of Theorem 2.2 in \cite{wiemeler_exotic_2013}, the fundamental group of \(M\) is isomorphic to \(\pi_1(\Sigma)=\mathbb{Z}_p\).
So for large \(p\), \(M\) does not admit almost non-negative sectional curvature by Corollary~\ref{sec:introduction-1}.

In the following we show that most of the previously known obstructions to the existence of almost non-negative sectional curvature cannot be used to get this conclusion.
(See the introduction of \cite{MR2630041} for a comprehensive list of these obstructions.)

First, the main result of \cite{MR2630041} shows that if there were almost non-negative curvature on \(M\) then \(\pi_1(M)\) contains a nilpotent subgroup of controlled finite index.
This is clearly satisfied by \(M\) because \(\pi_1(M)\) is a finite abelian group.

Since \(\pi_1(M)\) is cyclic, it clearly satisfies Gromov's estimate \cite{gromov_almost_1978} for the minimal number of generators of the fundamental group of an almost non-negatively curved manifold.

By \cite{wiemeler2} every torus manifold admits an invariant metric of positive scalar curvature.
In particular the Yamabe invariant of \(M\) is positive and its \(\hat{A}\)-genus vanishes in case \(M\) is spin.
So also the obstructions coming from \cite{schoen_variational_1989}, \cite{lebrun_ricci_2001}, \cite{gromov_volume_1982} and \cite{gallot_egalites_1983} cannot be used.

We also check that the manifold \(M\) itself satisfies Gromov's Betti number estimate \cite{gromov_curvature_1981}  for almost non-negatively curved manifolds.
This can be done by using Mayer-Vietoris-sequences.
To do so note that for a field \(K\) we have
\begin{equation*}
  \sum_{i\geq 0}\dim_K H^i(\Sigma;K)=
  \begin{cases}
    2n+2&\text{ if } \characteristic K=p\\
    2&\text{ otherwise. }
  \end{cases}
\end{equation*}
Hence we have
\begin{align*}
  \sum_{i\geq 0}\dim_K H^i(\Sigma'\times T^{2n+1};K)&\leq (2n+2) \cdot 2^{2n+1}
  \\ \sum_{i\geq 0}\dim_K H^i(S^{2n}\times T^{2n+1};K)&= 2 \cdot 2^{2n+1}
  \\\sum_{i\geq 0}\dim_K H^i(S';K)&\leq 2^{2n+1}+2
\end{align*}
by the Künneth theorem and the Mayer-Vietoris sequence for \(S=S'\cup (T^{2n+1}\times D^{2n+1})\).
Therefore from the Mayer-Vietoris sequence for \(M=S'\cup (\Sigma'\times T^{2n+1})\), we get
\begin{equation*}
  \begin{split}
     \sum_{i\geq 0} \dim_K H^i(M;K) & \leq  \sum_{i\geq 0} \dim_K H^i(\Sigma'\times T^{2n+1};K) + \sum_{i\geq 0}\dim_K H^i(S';K)\\ & + \sum_{i\geq 0}\dim_K H^i(S^{2n}\times T^{2n+1};K)\\ &\leq (2n+5)\cdot 2^{2n+1} +2.
  \end{split}
 \end{equation*}

 So for \(n\geq 1\) the sum of Betti numbers of \(M\) is smaller than \(2^{4n+2}\) which is both the sum of Betti numbers of a \((4n+2)\)-dimensional torus and the conjectural optimal upper bound for the sum of Betti numbers of an almost non-negatively curved manifold.
 
However, as one sees by tracing back the proof of Corollary~\ref{sec:introduction-1}, the universal covering space of \(M\) does not satisfy Gromov's Betti number estimate for fixed \(n\) and large \(p\).

\subsection{{Corollary~\ref{sec:introduction-1} as an obstruction to certain torus actions}}

Let \(M\) be a product of two lens spaces, then \(M\) is even-dimensional and admits a metric of non-negative sectional curvature.
But if the fundamental group of \(M\) is large then, by Corollary~\ref{sec:introduction-1} there is no torus action with at least one isolated fixed point on \(M\).
However there is an isometric action of a torus without any fixed points on \(M\).

\subsection{The examples of Bru\`e, Naber and Semola}
\label{sec:example-BNS}

In this section we review the construction of the examples of Bru\`e, Naber and Semola \cite{brue_compact_2026}.
While Theorem~\ref{sec:introduction-2} does not directly apply to these manifolds, we show that there is a torus action on these examples to which one can apply Lemma~\ref{sec:main-technical-lemma}.
From this we deduce that these examples do not admit almost non-negative sectional curvature.

These examples are \(9\)-dimensional manifolds \(N_k\), \(k\in\mathbb{N}\), with fundamental group a \(\mathbb{Z}_2\)-extension of finite Heisenberg groups \(\Gamma_k=H_3(\mathbb{Z}_k)\).
For simplicity in the following we assume that \(k\) is odd.
Then \(\mathbb{Z}_2\times \Gamma_k\) is the only such extension up to isomorphism because the order of \(\Gamma_k\) is odd.

The construction of the universal coverings \(\tilde{N}_k\) of \(N_k\) starts with a principal \(S^1\)-bundle \(\mathcal{N}_k\rightarrow \mathcal{B}_k\) over the \(3\)-sphere with \(2k\) open discs \(D^3\) removed such that the restriction to the boundary components of \(\mathcal{B}_k\) is given by the Hopf-fibration.

So the boundary components of \(\mathcal{N}_k\) are \(S^1\)-equivariantly diffeomorphic to \(S^3\) with the standard linear free \(S^1\)-action.
Hence one can glue to each boundary component a copy of \(D^4\) to get a four-dimensional spin-manifold \(M_k\) with semi-free \(S^1\)-action.
Since the fixed points of this \(S^1\)-action are isolated, it is of even type, i.e. it is compatible with the \(\text{Spin}\)-structure.

It is shown in the proof of Lemma 4.4 of \cite{brue_compact_2026} that there is an effective action of \(\Gamma_k\) on \(M_k\) which commutes with the \(S^1\)-action on \(M_k\).
As this action is not free they further lift the \(\Gamma_k\times S^1\)-action to the orthogonal frame bundle \(FM_k\) associated to the tangent bundle of \(M_k\).

Since \(M_k\) is spin there is a two-fold covering \(P_{\text{Spin}}M_k\rightarrow FM_k\) by a principal \(\text{Spin}(4)\)-bundle \(P_{\text{Spin}}M_k\) over \(M_k\).
Because \(k\) is odd the \(\Gamma_k\times S^1\)-action on \(FM_k\) lifts to an action of  \(\mathbb{Z}_2\times \Gamma_k\times S^1\) on \(P_{\text{Spin}}M_k\).
This action commutes with the action of a maximal torus \(T^2\subset \text{Spin}(4)\) on \(P_{\text{Spin}}M_k\).

Then \(\tilde{N}_k\) is defined as \(P_{\text{Spin}}M_k/H\), where \(H\) is a suitable circle subgroup of \(T^2\) such that the induced \(\mathbb{Z}_2\times \Gamma_k\)-action on \(\tilde{N}_k\) is free.
So one can define \(N_k=\tilde{N}_k/(\mathbb{Z}_2\times \Gamma_k)\).

On \(\tilde{N}_k\) there is an effective \(S^1\times (T^2/H)\)-action which commutes with the \(\mathbb{Z}_2\times \Gamma_k\)-action and is induced from the action on \(P_{\text{Spin}}M_k\).
Note that \(T^2/H\) acts freely on \(\tilde{N}_k\). 
Therefore it is easy to see that the \(S^1\times (T^2/H)\)-action restricted to the preimage of the open principal stratum of the \(S^1\)-action on \(M_k\) is free.

Let \(x\in M_k\) be an \(S^1\)-fixed point.
Then \(S^1\times (T^2/H)\) acts on the fiber \(\text{Spin}(4)/H\)  over \(x\) of \(\tilde{N}_k\rightarrow M_k\), such that the following holds:
\begin{enumerate}
\item \(T^2/H\) acts by multiplication from the right on \(\text{Spin}(4)/H\).
\item \(S^1\) acts by multiplication from the left on \(\text{Spin}(4)/H\) via a lift \(S^1\rightarrow \text{Spin}(4)\) of the isotropy representation \(S^1\rightarrow SO(4)\) on \(T_xM_k\).
\end{enumerate}
Since the Euler characteristic of \(\text{Spin}(4)/T^2\) is \(4\neq 0\), there must be a one-dimensional \(S^1\times (T^2/H)\)-orbit in \(\text{Spin}(4)/H\).

So there exists a fixed point component \(\tilde{F}\subset \text{Spin}(4)/H\) of some circle subgroup \(K\subset S^1\times (T^2/H)\).
Denote by \(\Gamma_{\tilde{F}}\) the subgroup of \(\mathbb{Z}_2\times \Gamma_k\) which leaves \(\tilde{F}\) invariant.
Let also \(F\) be the image of \(\tilde{F}\) in \(N_k\).

Then \(F\) is a fixed point component of the induced \(K\)-action on \(N_k\) and the image of \(\pi_1(F)\rightarrow \pi_1(N_k)\) can be identified with \(\Gamma_{\tilde{F}}\).
Note that \(\Gamma_{\tilde{F}}\) is a finite subgroup of \(\mathbb{Z}_2\times\Gamma_{kx}\) where \(\Gamma_{kx}\) is the subgroup of \(\Gamma_k\) which fixes \(x\in M_k\).

Since the \(\Gamma_k\)-action on \(M_k\) is effective and commutes with the semi-free \(S^1\)-action, \(\Gamma_{kx}\) can be identified with a finite subgroup of \(U(2)\) via the isotropy representation on \(T_xM_k\).
Hence it follows from Jordan's Theorem that \(\Gamma_{\tilde{F}}\) contains an abelian subgroup of uniformly bounded finite index.
Therefore by our Lemma~\ref{sec:main-technical-lemma} and Theorem~\ref{sec:almost-non-negative-1} \(N_k\) does not admit almost non-negative sectional curvature for \(k\) large, since otherwise the \(\mathbb{Z}_2\times\Gamma_k\) would contain an abelian subgroup of uniformly bounded finite index.

Again the actual obstruction to almost non-negative sectional curvature on \(N_k\) is that the total Betti numbers of the universal coverings \(\tilde{N}_k\) are unbounded.
One can also establish this using other methods.

Indeed, \(\tilde{N}_k\) is a fiber bundle with simply connected fiber over the simply connected manifold \(M_k\).
Moreover, the second Betti number of \(M_k\) is equal to
\[b_2(M_k)=\chi(M_k)-2=|M_k^{S^1}|-2=2(k-1).\]
Hence from the Hurewicz theorem and the long exact homotopy sequence for the fibration \(\tilde{N}_k\rightarrow M_k\) one gets that
\[b_2(\tilde{N}_k)=\dim \pi_2(\tilde{N}_k)\otimes \mathbb{Q} \geq \dim \pi_2(M_k)\otimes \mathbb{Q}=b_2(M_k)=2(k-1)\]
is unbounded.

\bibliographystyle{amsalpha}
\bibliography{torus_ricci}

\end{document}